\documentclass[reqno]{amsart}
\usepackage{amsmath}
\usepackage{amsthm}
\usepackage{amssymb}
\usepackage{stmaryrd}

\usepackage[utf8]{inputenc}
\usepackage[T1]{fontenc}
\usepackage{lmodern}
\usepackage{eucal}
\pdfoutput=1 
\usepackage{microtype}
\usepackage{url}
\usepackage{enumitem}
\usepackage{tikz-cd}
\usepackage{bbm}
\usepackage{longtable,booktabs}

\usepackage[letterpaper, left = 1.5in, right = 1.5in, top = 1.5in, bottom = 1.5in]{geometry}

\swapnumbers

\usepackage{hyperref}

\usepackage{xcolor}
\hypersetup{
    colorlinks,
    linkcolor={red!50!black},
    citecolor={blue!50!black},
    urlcolor={blue!80!black}
}

\usepackage[capitalise,nameinlink]{cleveref}
\crefname{subsection}{Subsection}{Subsections}
\crefname{subsubsection}{Subsubsection}{Subsubsections}

\theoremstyle{definition}
\newtheorem{theorem}{Theorem}[subsection]
\newtheorem{defn}[theorem]{Definition}

\newtheorem{ex}[theorem]{Example}
\newtheorem{cor}[theorem]{Corollary}
\newtheorem{lemma}[theorem]{Lemma}
\newtheorem{prop}[theorem]{Proposition}
\newtheorem{rmk}[theorem]{Remark}

\newtheorem{construction}[theorem]{Construction}
\newtheorem*{rmk*}{Remark}
\newtheorem*{ex*}{Example}
\newtheorem*{theorem*}{Theorem}
\newtheorem*{defn*}{Definition}

\newcommand{\bbZ}{\mathbb{Z}}

\newcommand{\bbH}{\mathbb{H}}

\newcommand{\bbE}{\mathbb{E}}
\newcommand{\bbQ}{\mathbb{Q}}

\newcommand{\bbC}{\mathbb{C}}

\newcommand{\sfh}{\mathsf{h}}

\newcommand{\Fin}{\mathcal{F}\mathrm{in}}

\newcommand{\Ind}{\operatorname{Ind}}
\newcommand{\Res}{\operatorname{Res}}

\newcommand{\Aut}{\operatorname{Aut}}

\newcommand{\Fib}{\operatorname{Fib}}

\newcommand{\coker}{\operatorname{Coker}}

\newcommand{\im}{\operatorname{Im}}
\renewcommand{\ker}{\operatorname{Ker}}

\newcommand{\Sq}{\mathrm{Sq}}

\newcommand{\res}{\operatorname{res}}
\newcommand{\tr}{\operatorname{tr}}

\newcommand{\sgn}{\mathrm{sgn}}

\newcommand{\End}{\operatorname{End}}

\newcommand{\id}{\operatorname{id}}
\newcommand{\Ann}{\operatorname{Ann}}

\newcommand{\Pin}{\operatorname{Pin}}
\newcommand{\Dic}{\mathrm{Dic}}

\newcommand{\h}{\mathrm{h}}

\newcommand{\ab}{\mathrm{ab}}

\newcommand{\rat}{W}

\newcommand{\ol}{\overline}

\newcommand\xqed[1]{%
  \leavevmode\unskip\penalty9999 \hbox{}\nobreak\hfill
  \quad\hbox{#1}}
\newcommand\tqed{\xqed{$\triangleleft$}}

\makeatletter
\DeclareRobustCommand{\tvdots}{%
  \vbox{\baselineskip4\p@\lineskiplimit\z@\kern0\p@\hbox{.}\hbox{.}\hbox{.}}}
\makeatother

\newcommand{\rightsim}{\xrightarrow{\raisebox{-1pt}{\tiny{$\sim$}}}}

\makeatletter
\@namedef{subjclassname@2020}{\textup{2020} Mathematics Subject Classification}
\makeatother

\begin{document}

\title{Equivariant $v_{1,\vec{0}}$-self maps}
\author{William Balderrama}
\author{Yueshi Hou}
\author{Shangjie Zhang}

\begin{abstract}
Let $G$ be a cyclic $p$-group or generalized quaternion group, $X\in \pi_0 S_G$ be a virtual $G$-set, and $V$ be a fixed point free complex $G$-representation. Under conditions depending on the sizes of $G$, $X$, and $V$, we construct a self map $v\colon\Sigma^V C(X)_{(p)}\rightarrow C(X)_{(p)}$ on the cofiber of $X$ which induces an equivalence in $G$-equivariant $K$-theory. These are transchromatic $v_{1,\vec{0}}$-self maps, in the sense that they are lifts of classical $v_1$-self maps for which the telescope $\smash{C(X)_{(p)}[v^{-1}]}$ can have nonzero rational geometric fixed points.
\end{abstract}

\maketitle

\section{Introduction}

Fixing a prime $p$, the \textit{Morava $K$-theories} $K(n)$ are a sequence of $p$-local homology theories, starting with $K(0) = H\bbQ$ and $K(1) = KU/p$, with the property that if $X$ is a nonzero finite $p$-local spectrum then $\{0\leq n < \infty : K(n)_\ast X = 0\} = [0,n)$ for some $n$, called the \textit{type} of $X$. The \textit{periodicity theorem} of Hopkins and Smith \cite{hopkinssmith1998nilpotence} asserts that every type $n$ spectrum admits a \textit{$v_n$-self map}: a map $\Sigma^k X \rightarrow X$ inducing an isomorphism in $K(n)$-theory.

We are interested in the analogous phenomena within \textit{equivariant} stable homotopy theory, where the situation is significantly richer. If $G$ is a group acting on a finite complex $X$, then the geometric fixed points $\Phi^H X$ can have \textit{different} types as $H$ ranges over the subgroups of $G$. Determining the possible types of the geometric fixed points of a finite $G$-spectrum amounts to the study of the \textit{Balmer spectrum} of $G$-spectra \cite{balmersanders2016spectrum}. Complete answers have been obtained for abelian groups \cite{barthelhausmannnaumannnikolausnoelstapleton2019balmer}, as well as extraspecial $2$-groups and generalized quaternion groups \cite{kuhnlloyd2022chromatic}. We refer the reader to \cite{behrenscarlisle2024periodic} for a survey of this emerging field of \textit{chromatic equivariant homotopy theory}. 

Classically, $v_1$-self maps were first produced by Adams in his work on the $J$-homomorphism \cite{adams1966on}. The construction of $v_n$-self maps remains a topic of major interest, as they allow one to produce and control \textit{infinite periodic families} within the stable homotopy groups of spheres \cite{davismahowald1981v1v2, behrenspemmaraju2004existence, behrenshillhopkinsmahowald2008existence, bhattacharyaeggermahowald2017periodic, bhattacharyaegger2020class, behrensmahowaldquigley2023primary, belmontshimomura2023beta}. Equivariantly, one must allow for a notion of ``equivariant $v_n$-self map'' where $n$ is allowed to vary over the subgroups of $G$. Only a very small number of equivariant self maps have been constructed recently, by Bhattacharya--Guillou--Li \cite{bhattacharyaguillouli2022rmotivic} and Behrens--Carlisle \cite{behrenscarlisle2024periodic}, mostly for $G = C_2$.

Our goal in this paper is to establish the following family of equivariant $v_1$-self maps. Say that a $G$-representation $V$ is \textit{fixed point free} if $G$ acts freely on the unit sphere in $V$.

\begin{theorem}\label{thm:selfmap}
Let $G$ be a cyclic $p$-group or generalized quaternion group, of order $p^n$. Let $X \in \pi_0 S_G$ be a virtual $G$-set of virtual cardinality $p^tc$ with $p\nmid c$. Let $V$ be a fixed point free complex $G$-representation of complex dimension $p^kc(p-1)$, or $2^{k-1}c$ when $p=2$, with $p\nmid c$, and when $p=2$ assume that $k\geq 3$.

If $k+1\geq n+t$, or $k>n$ when $(p,t) = (2,1)$, then there is a map $\Sigma^V C(X)_{(p)}\rightarrow C(X)_{(p)}$ inducing an equivalence in $G$-equivariant $K$-theory. 
\tqed
\end{theorem}

Cyclic $p$-groups and generalized quaternion groups appear in \cref{thm:selfmap} as these are the only $p$-groups which admit a fixed point free representation.

\begin{ex}\label{ex:cpsm}
Let $\ol{\rho}_p^\bbC$ denote the reduced complex regular representation of $C_p$. Taking $G = C_{p^n}$ and $X = [C_{p^n}]$, we obtain $C_{p^n}$-equivariant \textit{$v_{1,\vec{0}}$-self maps}
\[
\Sigma^{p^ns\Ind_{C_p}^{C_{p^n}}(\ol{\rho}_p^\bbC)}C([C_{p^n}])_{(p)}\rightarrow C([C_{p^n}])_{(p)},
\]
for $s$ even if $(p,n) = (2,1)$, which are not nilpotent on any geometric fixed points. Here, ``$v_{1,\vec{0}}$'' refers to the fact that these are lifts of the classical $v_1$-self map on $C(p^n)$ which act as a $v_0$-self map on all nontrivial geometric fixed points (see \cref{prop:typecofiber}).
\tqed
\end{ex}

\begin{ex}
The smallest nontrivial example at $p=2$ is a $C_2$-equivariant $v_{1,0}$-self map 
\begin{equation}\label{eq:c2}
\Sigma^{8\sigma}C(\mathsf{h})_{(2)}\rightarrow C(\mathsf{h})_{(2)},
\end{equation}
where $\mathsf{h} = [C_2] = \tr_e^{C_2}(1)$. This map was suggested to exist by Bhattacharya--Guillou--Li \cite[Remark 5.3]{bhattacharyaguillouli2022rmotivic} and recently constructed independently by Behrens and Carlisle in \cite[Proposition 8.23]{behrenscarlisle2024periodic}. It has precursors in work of Crabb \cite{crabb1980z2,crabb1989periodicity}.
\tqed
\end{ex}

We refer the reader to \cref{sec:examples} for an extensive list and discussion of examples. From now on, we will consider the prime $p$ to be fixed and implicitly $p$-localize everything in sight.

\begin{rmk}
We point out two interesting differences between the behavior of these equivariant self maps and their classical counterparts with an example. Classically, the $v_1$-self map $\Sigma^8 C(2)\rightarrow C(2)$ extends to $\Sigma^8 C(2^4)\rightarrow C(2^4)$. Equivariantly,
\begin{enumerate}
\item The self map $\Sigma^{8\sigma}C(\sfh)\rightarrow C(\sfh)$ of \cref{eq:c2} does \textit{not} extend to a self map of $C(\sfh^4)$;
\item It \textit{does} extend to a $KU_{C_2}/2$-local equivalence $L_{KU_{C_2}/2}\Sigma^{8\sigma}C(\sfh^4)\simeq L_{KU_{C_2}/2}C(\sfh^4)$.
\end{enumerate}
Thus not all classical self maps lift to equivariant self maps, and equivariant self maps do not account for all equivariant $K(1)$-local periodicities. See \cref{ex:k1} for details. 
\tqed
\end{rmk}

This paper is organized as follows. In \cref{sec:adams}, we describe a general strategy for constructing $G$-equivariant $v_1$-self maps whenever $G$ is a $p$-group. Implementing this strategy takes some work, and we give the necessary technical preliminaries in \cref{sec:lemmas}. We then prove \cref{thm:selfmap} in \cref{sec:proof}. Finally, we end in \cref{sec:examples} with examples and some additional observations.

Our techniques in this paper have some overlap with \cite{balderramahouzhang2024cpn}, where we studied a $C_{p^n}$-equivariant generalization of the classical Mahowald invariant \cite{mahowaldravenel1993root}. Both papers can be seen as applications of equivariant $K$-theory to the equivariant stable stems. However, our work in this paper is more conceptual and significantly less technical, requiring almost nothing in the way of delicate algebraic computations or working with explicit cell structures. We recall the couple facts that we will need from \cite{balderramahouzhang2024cpn} in \cref{ssec:j}.

\section{Constructing equivariant \texorpdfstring{$v_1$}{v\_1}-self maps}\label{sec:adams}

The proof of \cref{thm:selfmap} is similar in form to Adams' original construction of $v_1$-self maps on Moore spectra. Our basic tool is the following variation of \cite[Lemma 12.5]{adams1966on}.

\begin{lemma}\label{lem:adams}
Let $\rho$ be a real $4k$ ($8k$ if $p=2$)-dimensional virtual $G$-representation, and let $X\in A(G)$ be a virtual $G$-set of virtual cardinality $|X| = p^tc$ with $p\nmid c$. Suppose given $\alpha \in \pi_{\rho-1}S_G$ satisfying the following conditions:
\begin{enumerate}
\item The image of $\alpha$ under
$
\pi_{\rho-1}S_G\rightarrow \pi_{|\rho|-1}S\rightarrow\pi_{|\rho|-1}J
$
has order exactly $p^t$, where $J = S_{K(1)} = \Fib(\psi^\ell -1 \colon KO_p^\wedge \rightarrow KO_p^\wedge)$ is the $p$-primary image-of-$J$ spectrum;
\item $X\cdot \alpha = 0$;
\item $0 \in \langle X,\alpha,X\rangle$.
\end{enumerate}
Then there exists a map $v$ filling in
\begin{center}\begin{tikzcd}
S^\rho\ar[r,"\alpha"]\ar[d]&S^1\\
\Sigma^\rho C(X)\ar[r,"v"]&C(X)\ar[u,"\partial"]
\end{tikzcd},\end{center}
and any such $v$ induces an isomorphism in $KU_G$-theory.
\end{lemma}
\begin{proof}
When $G = e$, this is exactly \cite[Lemma 12.5]{adams1966on}, only:
\begin{enumerate}[label=(\roman*)]
\item We work $p$-locally, and thus reduce to the case $m=p^t$.
\item We replace the condition $e_\bbC(\alpha) =-\frac{1}{p^t}$ with the condition $\alpha$ has order $p^t$ in $J$.
\item The replacement in (ii) only changes things by a unit, and so we replace $d_\bbC(v) = 1$ with just the condition that $v$ induces an isomorphism in $K$-theory.
\end{enumerate}

In general, conditions (2) and (3) enable one to construct the diagram, and by the above, (1) ensures that $KU_G\otimes v$ is an equivalence on underlying spectra. It follows that $KU_G\otimes v$ is an equivalence after Borel completion. As $G$ is a $p$-group, the Atiyah--Segal completion theorem implies $KU_G\otimes v$ is an equivalence after $p$-completion \cite[\S III]{atiyahtall1969group}. It is thus an equivalence already, as $p$-completion is conservative on finitely generated $\bbZ_{(p)}$-modules.
\end{proof}

There are three steps in applying \cref{lem:adams}, corresponding to the three conditions:
\begin{enumerate}
\item Construct a ``candidate'' $\alpha \in \pi_{\rho-1}S_G$ satisfying (i);
\item Verify that this candidate satisfies $X\cdot \alpha = 0$;
\item Verify that $0 \in \langle X,\alpha,X\rangle$.
\end{enumerate}

Classically, when $G = e$, the positive resolution of the Adams conjecture implies that $\pi_{4k-1}S\rightarrow \pi_{4k-1}J$ is a split surjection, with splitting arising from the $J$-homomorphism:
\begin{center}\begin{tikzcd}
&\im J\ar[d,tail]\ar[dr,"\cong"] \\
\pi_{4k}KO\ar[r,"J"]\ar[ur]&\pi_{4k-1}S\ar[r, two heads]&\pi_{4k-1}J
\end{tikzcd}.\end{center}
Thus good choices of $\alpha$ are given by suitable elements in the image of $J$, and exactly what elements exist to give self maps of $C(p^t)$ is entirely controlled by the structure of $\pi_\ast J$.

Equivariantly, good candidates $\alpha \in \pi_{\rho-1}S_G$ are given by equivariant lifts of classes in the image of $J$. However, for an arbitrary virtual representation $\rho$, there is no natural $J$-homomorphism valued in $\pi_{\rho-1}S_G$, though in some cases the $\tilde{J}$ construction studied in \cite{balderrama2023equivalences} applies. Thus more work is needed to construct the candidate classes for step (1), and once they are constructed one must still carry out steps (2) and (3).

In particular, the existence and order of equivariant $v_1$-self maps is \textit{not} detected $K(1)$-locally. Specifically, define $J_G = L_{KU_G/p} S_G$. This is a $G$-equivariant analogue of the $K(1)$-local sphere, and fits in a fiber sequence of the form
\begin{equation}\label{eq:jg}
\begin{tikzcd}
J_G\ar[r]&(KO_G)_p^\wedge\ar[r,"\psi^\ell-1"]&(KO_G)_p^\wedge
\end{tikzcd}.\end{equation}
As $J_G$ is reasonably computable, one would like to reduce the construction of equivariant $v_1$-self maps to computations in $J_G$-theory, just as can be done nonequivariantly. Unfortunately, $\pi_{\rho-1}S_G\rightarrow\pi_{\rho-1}J_G$ fails to be a surjection in general, and thus the self maps on $C(X)$ that one might predict from the structure of $\pi_\star J_G$ may fail to actually exist.

\begin{ex}\label{ex:k1}
The computations of \cite{balderrama2021borel} show that the generator $\sigma \in \pi_7 J_2$ lifts to a class $\tilde{\sigma} \in \pi_{8\sigma-1}J_{C_2}$ satisfying $\mathsf{h}^4\cdot \tilde{\sigma} = 0$. Here $\mathsf{h} = [C_2] \in \pi_0 S_{C_2}$ satisfies $\mathsf{h}^2 = 2\mathsf{h}$. Applying \cref{lem:adams} in the $KU_{C_2}/2$-local category, one obtains an equivalence
\[
\Sigma^{8\sigma}L_{KU_{C_2}/2}C(\mathsf{h}^4)\rightsim L_{KU_{C_2}/2}C(\mathsf{h}^4)
\]
lifting the $K(1)$-localization of the classical $v_1$-self map on $C(2^4)$.

By contrast, Araki and Iriye's work on the $C_2$-equivariant stable stems \cite[Theorem 14.18(v)]{arakiiriye1982equivariant} shows that $\sigma\in \pi_7 S$ does \textit{not} lift to $\pi_{8\sigma-1}S$. Instead, one only has $\tr(\sigma) \in \pi_{8\sigma-1}S$, lifting $2\sigma$ and satisfying $\mathsf{h}^3\cdot \tr(\sigma) = 0$. Applying \cref{lem:adams}, this gives a $KU_{C_2}$-equivalence
\[
\Sigma^{8\sigma}C(\mathsf{h}^3)\rightarrow C(\mathsf{h}^3)
\]
which lifts the classical $v_1$-self map on $C(2^3)$ and does not extend to $C(\mathsf{h}^4)$.
\tqed
\end{ex}

\section{Technical lemmas and propositions}\label{sec:lemmas}

The strategy for constructing equivariant $v_1$-self maps outlined in \cref{sec:adams} requires constructing and controlling certain elements in the equivariant stable stems. This section establishes the technical lemmas and propositions that we will need in order to carry this out in  \cref{sec:proof}.

For this section only, we drop the convention that everything has been implicitly localized at a prime $p$, as much of what we wish to say holds integrally.

\subsection{Representations of cyclic and dicyclic groups}\label{ssec:reptheory}

We begin by recalling some facts about the representation theory of the cyclic and dicyclic groups. In particular, we will show that \cref{thm:selfmap} is equivalent to the following:

\begin{theorem}\label{thm:mainthm2}
Let $G$ be a $p$-group, of order $p^n$. Let $X \in \pi_0 S_G$ be a virtual $G$-set of virtual cardinality $p^tc$ with $p\nmid c$. Let $V$ be a fixed point free complex $G$-representation with rational characters, of complex dimension $p^kc(p-1)$, or $2^{k-1}c$ when $p=2$, with $p\nmid c$, and when $p=2$ assume that $k\geq 3$.

If $k+1\geq n+t$, or $k>n$ when $(p,t) = (2,1)$, then there is a map $\Sigma^V C(X)_{(p)}\rightarrow C(X)_{(p)}$ inducing an equivalence in $G$-equivariant $K$-theory. 
\tqed
\end{theorem}

All $p$-groups which admit a fixed point free representation are cyclic or quaternionic, so we must explain why it is sufficient to consider representations with rational characters.

\subsubsection{Cyclic groups}

We begin with the cyclic groups. Write $T = U(1)$ for the circle group, with tautological complex character $L$, so that $RU(T) \cong \bbZ[L^{\pm 1}]$. The $k$th power map on $\bbC$ defines an $T$-equivariant map $S(L)\rightarrow S(L^k)$ on unit spheres, suspending to
\begin{equation*}\label{eq:kthpower}
\psi_k\colon S^{L}\rightarrow S^{L^{k}}
\end{equation*}
satisfying
\begin{equation}\label{eq:powermap}
\Phi^{C_d}\psi_k = \begin{cases}
k\colon S^2\rightarrow S^2 & d = 1,\\
0\colon S^0\rightarrow S^2 & 1\neq d \mid k,\\
1\colon S^2\rightarrow S^2 & d\nmid k.
\end{cases}
\end{equation}
Here, if $m$ is a positive integer, then
\[
C_m = \langle e^{2\pi i/m}\rangle \subset T
\]
is the cyclic subgroup of order $m$. Write
\[
\rat_m = \sum\{L^k : 1\leq k \leq m,\, \gcd(m,k) = 1\}
\]
for the fixed point free irreducible rational representation of $C_m$, of dimension equal to the value of the Euler totient function $\phi(m)$.

\begin{lemma}\label{lem:reducetorational1}
If $V$ is a fixed point free complex $C_m$-representation of dimension $k\phi(m)$, then there is an equivalence $S^V_{(m)}\simeq S^{k\rat_m}_{(m)}$ which is the identity on all nonzero geometric fixed points.
\end{lemma}
\begin{proof}
If $L_1$ and $L_2$ are faithful complex characters of $C_m$, then $L_1 = L_2^d$ for some $d$ coprime to $m$, and so the map $\psi_d\colon S^{L_2}[\tfrac{1}{d}]\rightarrow S^{L_1}[\tfrac{1}{d}]$ discussed in the previous section does the job. The lemma follows by smashing such equivalences together.
\end{proof}

Restricting to the case where $m$ is a power of $p$, we find that
\[
\rat_p = \ol{\rho}_p^\bbC = L + \cdots + L^{p-1}
\]
is the reduced complex regular representation of $C_p$, and
\[
\rat_{p^n} = \tr_{C_p}^{C_{p^n}}(\ol{\rho}_p^\bbC) = \bbC[C_{p^n}] - \bbC[C_{p^n}/C_p]
\]
is $p^{n-1}(p-1)$-dimensional. In particular, \cref{lem:reducetorational1} reduces \cref{thm:selfmap} to \cref{thm:mainthm2} when $G$ is cyclic, as the dimension of any $V$ satisfying the hypotheses of \cref{thm:selfmap} is a multiple of $p^{n-1}(p-1)$.

\subsubsection{Dicyclic groups}

The case of dicyclic groups is similar, only with the circle group $T$ replaced by 
\[
\Pin(2) = T\amalg j\cdot T\subset Sp(1).
\]
Among the finite subgroups of $\Pin(2)$ are the dicyclic groups
\[
\Dic_m = \langle e^{\pi i/m},j\rangle\subset \Pin(2).
\]
These groups can be understood as fitting into extensions
\begin{center}\begin{tikzcd}
1\ar[r]&C_{2m}\ar[r]\ar[d]&\Dic_{m}\ar[r]\ar[d]&C_2\ar[d]\ar[r]&1\\
1\ar[r]&S^1\ar[r]&\Pin(2)\ar[r]&C_2\ar[r]&1
\end{tikzcd}.\end{center}
The generalized quaternion groups are $Q_{2^n} = \Dic_{2^{n-2}}$ for $n\geq 3$. Observe that
\[
\bbZ\{\psi^\ell H : \ell \geq 0\}\subseteq RU(\Pin(2)),
\]
where $H$ is the tautological quaternionic representation, satisfying
\[
\Res^{\Pin(2)}_{T}(H)\cong L + L^{-1},\qquad H\cong \Ind_{T}^{\Pin(2)}(L).
\]
The representation $\psi^\ell H$ restricts to a fixed point free representation of $\Dic_m$ provided $\gcd(2m,\ell) = 1$, and these are the only fixed point free representations of $\Dic_m$. They are all Galois conjugate, and it follows that all fixed point free $\Dic_m$-representations with rational characters are a multiple of the $\phi(2m)$-dimensional representation
\[
H_{m} = \sum\{\psi^\ell H : 1\leq \ell \leq m,\, \gcd(\ell,2m) = 1\} = \tfrac{1}{2}\Ind_{C_{2m}}^{\Dic_m}\rat_{2m}.
\]
The identification $H = \Ind_{T}^{\Pin(2)}(L)$ implies that \cref{eq:powermap} norms to a $\Pin(2)$-equivariant map
\[
N_T^{\Pin(2)}\psi_d\colon S^H\rightarrow S^{\psi^d H}
\]
which is an equivalence after inverting $d$ and restricting to $\Dic_m$ for $\gcd(d,2m) = 1$. In particular, we have the following.

\begin{lemma}\label{lem:reducetorational2}
If $V$ is a fixed point free complex $\Dic_m$-representation of dimension $k\phi(2m)$, then there is an equivalence $S^V_{(2m)}\simeq S^{k H_{m}}_{(2m)}$.
\qed
\end{lemma}

\subsection{Equivariant \texorpdfstring{$J$}{J}-theory}\label{ssec:j}

We need the following facts about equivariant $K$-theory that were established in \cite{balderramahouzhang2024cpn}. Recall the definition of $J_G$ from \cref{eq:jg}.

\begin{prop}[{\cite[Proposition 2.2.3]{balderramahouzhang2024cpn}}]\label{prop:adamsconj}
Let $T$ be a compact pointed $G$-space. If $p>2$ or the Borel construction $T_{\h G} = EG_+\wedge_G T$ is simply connected, then
\[
\bbZ_p \otimes \pi_0^G D(\Sigma^\infty T) \rightarrow \widetilde{J}_G^0(T)
\]
is a naturally split surjection, this naturality being in maps of pointed $G$-spaces.
\qed
\end{prop}

\begin{prop}[{\cite[Corollary 3.3.4]{balderramahouzhang2024cpn}}]\label{lem:adamsop}
Fix a $p$-group $G$ of order $p^n$ and a fixed point free complex $G$-representation $V$ with rational characters, of complex dimension $p^km(p-1)$ (or $2^{k-1}m$ with $k\geq 2$ for $p=2$) with $p\nmid m$. If $p\nmid \ell$, then the action of the Adams operation $\psi^\ell$ on the Bott class $\beta_V \in \pi_V KU_G$ satisfies

\[
(\psi^\ell-\id)(\beta_V) = d\cdot p^{k+1-n}\cdot \tr_e^G(1)\cdot \beta_V
\]
for $d$ some integer coprime to $p$.
\qed
\end{prop}

We also need the following to safely apply \cref{prop:adamsconj}.

\begin{lemma}\label{lem:suspension}
Let $G$ be a finite group and $V$ be a nonzero $G$-representation in which $G$ embeds. Then $F(SG,S^V)$ is an equivariant suspension spectrum, i.e.\ $F(SG,S^V)\simeq \Sigma^\infty T$ for a compact pointed $G$-space $T$. Moreover, this equivalence may be chosen so that composite
\[
\Sigma^\infty(G_+\wedge S^{|V|-1})\simeq F(\Sigma G_+,S^V)\rightarrow F(SG,S^V)\simeq \Sigma^\infty T
\]
is induced by a map of spaces.
\end{lemma}
\begin{proof}
After fixing an embedding $G\rightarrow S(V)$, equivariant Spanier--Whitehead duality allows us to identify
\[
F(SG,S^V)\simeq S(S(V)\setminus G)\simeq S^V\setminus SG,
\]
so we may take $T = S^V\setminus G$. Now the composite
\[
\Sigma^\infty(G_+\wedge S^{|V|-1})\simeq F(\Sigma G_+,S^V)\rightarrow F(SG,S^V)\simeq \Sigma^\infty Z
\]
is adjoint to a nonequivariant map
\[
S^{|V|-1}\rightarrow \bigvee_{|G|-1}S^{|V|-1}.
\]
If $|V|\geq 2$, then any such map is induced by a map of spaces. If $|V| = 1$, then necessarily $G = C_2$ and $V = \sigma$, in which case this is the identity $S^0\rightarrow S^0$.
\end{proof}

\subsection{Symmetric Toda brackets}

In \cite[Corollary 3.7]{toda1962composition}, Toda proved that if $r$ is an integer and $y\in \pi_n S$ satisfies $ry = 0$, then
\[
\langle r,y,r\rangle \ni \begin{cases}0&r\not\equiv 2 \pmod{4},\\
\eta y & r \equiv 2 \pmod{4}.
\end{cases}
\]
In verifying condition (3) of \cref{lem:adams}, one would like to have a generalization of this that applies to the equivariant stable stems. Such a generalization is provided by work of Baues and Muro in \cite{bauesmuro2009toda}, as we now describe.

We begin by recalling the construction of the power operation $\Sq_1$. We essentially follow the treatment in \cite[Section 1.6.4]{lurie2018ellipticii}.

\begin{construction}\label{constr:sq1}
Given an $\bbE_\infty$ space $Z$ (not necessarily grouplike), the operation
\[
\pi_0 Z \ni z \mapsto \Sq_1(z) \in \pi_1(Z;z^2)
\]
is defined as follows. Fix $z\in \pi_0 Z$, and consider the following diagram:
\begin{center}\begin{tikzcd}
\ast\ar[r,"z"]\ar[d]&Z\\
\coprod_n B\Sigma_n\ar[ur,dashed]&B\Sigma_2\ar[l]&S^1\ar[l]\ar[ul,"\Sq_1(z)"']
\end{tikzcd}.\end{center}
As $Z$ is $\bbE_\infty$, the map $z\colon \ast\rightarrow Z$ extends canonically through the free $\bbE_\infty$ space $\coprod_n B\Sigma_n$. Now $\Sq_1(z) \in \pi_1(Z;z^2)$ is defined as the image of the unique generator of $\pi_1 B\Sigma_2\cong \bbZ/(2)$ under this map.
\tqed
\end{construction}

\begin{defn}
Given a $\bbE_\infty$ ring $R$, the (nonadditive) operation
\[
\Sq_1\colon \pi_0 R\rightarrow \pi_1 R
\]
is given by applying \cref{constr:sq1} to the multiplicative $\bbE_\infty$ structure on $\Omega^\infty R$, and further using the translation isomorphisms $\pi_1(\Omega^\infty R;r)\cong \pi_1 R$ defined for any $r\in \pi_0 R$.
\tqed
\end{defn}

The relevance of $\Sq_1$ to symmetric Toda brackets is the following.

\begin{prop}\label{lem:sq1}
Let $R$ be an $\bbE_\infty$ ring within $G$-spectra, and let $\rho$ be a virtual $G$-representation. Suppose given $x\in \pi_0 R$ and $y\in \pi_\rho R$ satisfying $xy = 0$. Then
\[
\Sq_1(x)\cdot y \in \langle x,y,x\rangle.
\]
\end{prop}
\begin{proof}
When $R$ is a nonequivariant connective ring spectrum, this is exactly \cite[Theorem 1.9]{bauesmuro2009toda}. The general case follows from this by considering $\left((R \oplus \Sigma^{-\rho} R)^G\right)_{\geq 0}$, where $R\oplus \Omega^\rho R$ is given an $\bbE_\infty$ ring structure as a square-zero extension of $R$.
\end{proof}

As far as our proof of \cref{thm:selfmap} is concerned, we only need the following example.

\begin{lemma}\label{lem:sq1int}
For an integer $n$, we have
\[
\Sq_1(n) = \begin{cases}0&n\equiv 0,1\pmod{4},\\ \eta&n\equiv 2,3\pmod{4}.\end{cases}
\]
\end{lemma}
\begin{proof}
Induct using $\Sq_1(0) = 0 = \Sq_1(1)$ and the identity \cite[(T11)]{bauesmuro2009toda}
\[
\Sq_1(a+b) = \Sq_1(a)+\Sq_1(b)+\eta \cdot a \cdot b.\qedhere
\]
\end{proof}

We round out this discussion by explaining how to compute
\[
\Sq_1\colon \pi_0 S_G\rightarrow \pi_1 S_G
\]
for a finite group $G$. This is not needed for our proof of \cref{thm:selfmap}, but we expect it to be useful for future equivariant computations.

We begin by recalling the structure of $\pi_1 S_G$. As $G$ is a finite group, $(\Omega^\infty S_G)^G$ is equivalent to the group completion of the groupoid $\Fin_G^\simeq$ of finite $G$-sets. The groupoid of finite $G$-sets is equivalent to the coproduct, over conjugacy classes of subgroups $H\subset G$, of the free coproduct-completions of the $1$-object categories corresponding to the Weyl groups $W_GH$. Associated to this is the tom Dieck Splitting
\[
(S_G)^G\simeq \bigoplus_{(H)}\Sigma^\infty_+ BW_GH,
\]
and we can use this splitting to identify
\[
\pi_1 S_G \cong \bigoplus_{(H)}\left(\{\pm 1\} \times (W_GH)^\ab\right).
\]

The Cartesian product of $G$-sets makes $\Fin_G^\simeq$ into a symmetric monoidal groupoid, i.e.\ a $1$-truncated $\bbE_\infty$ space. If $X$ is a finite $G$-set, then we may view $X$ as a path component of this space, and \cref{constr:sq1} specializes to
\[
\Sq_1(X) = \tau \in \Aut_{\Fin_G}(X\times X),\qquad \tau(a,b) = (b,a).
\]
Group completion induces a map of $\bbE_\infty$ spaces
\[
\phi\colon \Fin_G^\simeq\rightarrow (\Omega^\infty S_G)^G,
\]
and, by naturality, to describe $\Sq_1(X) \in \pi_1 S_G$ it therefore suffices to understand the effect of $\phi$ on $\pi_1$. If $T$ is a finite $g$-set, then this is a map
\[
\phi_T\colon \Aut_{\Fin_G}(T)\rightarrow \pi_1 S_G
\]
which can be computed as follows. First split $T$ into a coproduct of orbits, lumping orbits of the same type together
\[
T \cong \coprod_{(H)}(n_H\cdot G/H),
\]
so that
\[
\Aut_{\Fin_G}(T)\cong \prod_{(H)}\Sigma_{n_H}\wr W_G(H).
\]
Now $\phi_T$ is just the abelianization map
\begin{gather*}
\prod_{(H)}\Sigma_{n_H}\wr W_GH \rightarrow \prod_{(H)}\left(\{\pm 1\}\times (W_GH)^\ab\right),\\ (\sigma_H,(x_{H,1},\ldots,x_{H,n_H}))_{(H)} \mapsto (\sgn(\sigma_H),x_{H,1}\cdots x_{H,n_H})_{(H)}.
\end{gather*}

\begin{ex}\label{ex:sq1g}
Consider $G = G/e$ as a $G$-set over itself. Then there is an isomorphism
\[
\coprod_{x\in G}G\cdot \iota_x\rightsim G\times G,\qquad \iota_x \mapsto (e,x),
\]
and so $\Sq_1([G/e])$ lives in the summand $\{\pm 1 \}\times G^\ab \cong \pi_1 \Sigma^\infty_+ BG\subset \pi_1 S_G$. Under this isomorphism, the swapping map $\tau\colon G\times G\rightarrow G\times G$ can be identified as
\[
\coprod_{x\in G}G\cdot \iota_x\rightarrow\coprod_{x\in G}G\cdot\iota_x,,\qquad \iota_x\mapsto x\iota_{x^{-1}}. 
\]
It follows that
\[
\Sq_1([G/e]) = (\sgn(\sigma),\prod_{g\in G}g),
\]
where $\sigma\colon G\rightarrow G$ is the permutation given by inversion.
\tqed
\end{ex}

\begin{ex}
Specializing \cref{ex:sq1g} to the case where $G = C_n$ is a cyclic group of order $n$, say with generator $g$, and writing $\pi_1 \Sigma^\infty_+ BC_n \cong \bbZ/(2)\{\eta\cdot [C_n/e]\}\times C_n$, we find
\[
\Sq_1([C_n/e]) = \begin{cases}
(\eta\cdot [C_n/e],g^{n/2})&n\equiv 0 \pmod{4},\\
(0,e)&n\equiv 1\pmod{4},\\
(0,g^{n/2})&n\equiv 2 \pmod{4},\\
(\eta\cdot [C_n/e],e)&n\equiv 3\pmod{4}.
\end{cases}
\]
When $G = C_2$, this is detected by $\tau h_1$ in the equivariant Adams spectral sequence \cite{guillouisaksen2024c2}.
\tqed
\end{ex}

\subsection{\texorpdfstring{$v_1$}{v\_1}-self maps and the Bott class}

This subsection addresses a technical point (which is not strictly needed for \cref{thm:selfmap}): the $KU_G$-equivalences produced by \cref{lem:adams} are \textit{not} guaranteed to induce multiplication by the Bott class in $KU_G$-theory. We resolve this issue in our particular case, at least when $G = C_{p^n}$ is a cyclic $p$-group. It seems plausible something more general could be said, but we shall not pursue this.

So fix a prime $p$, consider a cyclic $p$-group $C_{p^n}$ with generator $g$, and let $\Gamma = \Aut(C_{p^n})$. The action of $\Gamma$ on $C_{p^n}$ has orbits
\[
O_i = \{g^k : \gcd(k,p^n) = p^i\}
\]
for $1\leq i \leq n$, and accordingly the action of $\Gamma$ on $RU(C_{p^n})\cong \bbZ[C_{p^n}]$ by Adams operations splits as a sum of permutation modules:
\[
RU(C_{p^n}) = \bigoplus_{1\leq i \leq n}M_i,\qquad M_i \cong \bbZ\{O_i\}.
\]
The representations $\gamma_i = \sum_{x\in O_i}x$ form a basis for the fixed points $RU(C_{p^n})^\Gamma$. Moreover, $\bbC[C_{p^n}/C_{p^i}] = \gamma_n+\cdots+\gamma_{i}$, realizing an isomorphism $A(C_{p^n})\cong RU(C_{p^n})^\Gamma$. We will use this to identify the Burnside ring $A(C_{p^n})$ as a subring of $RU(C_{p^n})$.

For an element $r$ of a ring $R$, write $\Ann(r;R) = \ker(r\colon R\rightarrow R)$.

\begin{lemma}\label{lem:fpiso}
Fix a virtual $C_{p^n}$-set $X\in A(C_{p^n})$. Then
\[
\Ann(X;A(C_{p^n}))\cong \Ann(X;RU(C_{p^n}))^\Gamma\quad\text{and}\quad A(C_{p^n})/X \cong (RU(C_{p^n})/X)^\Gamma.
\]
\end{lemma}
\begin{proof}
The first isomorphism is a formal consequence of $A(C_{p^n})\cong RU(C_{p^n})^\Gamma$.

For the second, we can write
\[
X = p^kc\left(\gamma_t+\sum_{1\leq i\neq t \leq n}n_i \gamma_i\right)
\]
for some $t$, where $n_i \in \bbZ_{(p)}$ and $c\in \bbZ_{(p)}^\times$. It follows that
\[
A(C_{p^n})/X = \bbZ\{\gamma_i : 1\leq i \neq t \leq n\} \oplus \bbZ/(p^k)\{X\}
\]
and that $RU(C_{p^n})/X$ fits into a short exact sequence
\[
0\rightarrow \bbZ/(p^k)\{X\}\oplus \bigoplus_{1\leq i \neq t \leq n}M_i\rightarrow RU(C_{p^n})/X \rightarrow M_t/\bbZ\{\gamma_t\}\rightarrow 0
\]
of $\Gamma$-modules. As $(M_t/\bbZ\{\gamma_t\})^\Gamma = 0$, it follows that 
\[
(RU(C_{p^n})/X)^\Gamma\cong \left(\bbZ/(p^k)\{X\}\oplus \bigoplus_{1\leq i \neq t \leq n}M_i\right)^\Gamma\cong A(C_{p^n})/X
\]
as claimed.
\end{proof}

\begin{prop}\label{prop:surjfp}
Fix a virtual $C_{p^n}$-set $X\in A(C_{p^n})$. Then 
\[
\pi_0 \End(C(X))\rightarrow \pi_0 KU_{C_{p^n}}\otimes \End(C(X))
\]
surjects onto the subring fixed by all Adams operations $\psi^\ell$ with $p\nmid \ell$.
\end{prop}
\begin{proof}
Consider the defining cofiber sequences
\begin{center}\begin{tikzcd}[row sep=0mm]
S_{C_{p^n}}\ar[r,"X"]&S_{C_{p^n}}\ar[r]&C(X)\\
\End(C(X))\ar[r]&C(X)\ar[r,"X"]&C(X)
\end{tikzcd}.\end{center}
As $\pi_1 S_{C_{p^n}}$ is finite and $\pi_1 KU_{C_{p^n}} = 0$, these show
\begin{align*}
\pi_0 \End(C(X))/\text{torsion} &\cong A(C_{p^n})/X\oplus \Ann(X;A(C_{p^n})),\\
\pi_0 KU_{C_{p^n}}\otimes \End(C(X))&\cong RU(C_{p^n})/X\oplus \Ann(X;RU(C_{p^n})),
\end{align*}
and so the proposition follows from \cref{lem:fpiso}.
\end{proof}

\begin{cor}\label{cor:bottclass}
The $KU_{C_{p^n}}$-equivalences guaranteed by \cref{thm:mainthm2} may be chosen to induce multiplication by the Bott class $\beta_V\in \pi_V KU_{C_{p^n}}$ in equivariant $K$-theory.
\end{cor}
\begin{proof}
Fix a $KU_{C_{p^n}}$-equivalence $v\colon \Sigma^V C(X)\rightarrow C(X)$ guaranteed by \cref{thm:mainthm2}. Consider the Hurewicz map
\[
h\colon \pi_\star \End(C(X))\rightarrow \pi_\star KU_{C_{p^n}}\otimes \End(C(X)).
\]
Considered as an element of $\pi_V\End(C(X))$, the map $v$ satisfies $h(v) = \beta_V\cdot u$ for some invertible element $u \in \pi_0 KU_{C_{p^n}}\otimes \End(C(X))$. Our hypotheses on $V$ and $X$ ensure that $\beta_V\in \pi_V KU_{C_{p^n}}\otimes\End(C(X))$ is fixed by Adams operations coprime to $p$. Being in the image of $h$, the same is true of $\beta_V\cdot u$, and thus of $u$. Hence $u^{-1} = h(u')$ for some $u' \in \pi_0 \End(C(X))$ by \cref{prop:surjfp}, and now $v' = u'\cdot v \colon \Sigma^V C(X)\rightarrow C(X)$ is a self map inducing multiplication by $\beta_V$ on the nose in equivariant $K$-theory.
\end{proof}

\section{The proof of \texorpdfstring{\cref{thm:selfmap}}{Theorem 1.0.1}}\label{sec:proof}

We are now in a position to implement the strategy outlined in \cref{sec:adams} to prove \cref{thm:selfmap}, in its equivalent reformulation as \cref{thm:mainthm2}. We begin by defining some classes. First,
\[
j\in \pi_{|V|-1}S
\]
is a generator of the image of $J$ in this degree, of order $p^{k+1}$. Second,
\[
\alpha = \tr_e^G(p^{k+1-n-t}j) \in \pi_{V-1}S_G.
\]
We now proceed to verify that $\alpha$ satisfies the three conditions of \cref{lem:adams}.

\subsection{Step 1}
We first verify that the image of $\alpha$ under $\pi_{V-1}S_G\rightarrow \pi_{|V|-1}S\rightarrow\pi_{|V|-1}J$ has order exactly $p^t$. Indeed, as $V$ is oriented, the class
\[
\res^G_e(\alpha) = \res^G_e(\tr_e^G(p^{k+1-n-t} j)) = |G| \cdot p^{k+1-n-t}j = p^{k+1-t}j
\]
is an element in the image of $J$ of order exactly $p^t$.

\subsection{Step 2}

We next verify that $X\cdot \alpha = 0$. By Frobenius reciprocity, we may identify
\[
X\cdot \alpha = X \cdot \tr_e^G(p^{k+1-n-t}j) = \tr_e^G(\res^G_e(X)p^{k+1-n-t}j) = \tr_e^G(p^{k+1-n}cj).
\]
As $c$ is invertible, it suffices to show that $\tr_e^G(p^{k+1-n}j) = 0$. Consider the cofiber sequence
\begin{center}\begin{tikzcd}
G_+\ar[r]&S^0\ar[r]&SG
\end{tikzcd},\end{center}
with associated long exact sequence
\begin{equation}\label{eq:transferles}
\begin{tikzcd}
\cdots\ar[r]&\pi_V SG\ar[r,"r"]&\pi_{|V|-1}S\ar[r,"\tr"]&\pi_{V-1}S_G\ar[r]&\cdots
\end{tikzcd}.\end{equation}
By this long exact sequence, to show that $\tr_e^G(p^{k+1-n}j) = 0$, it suffices to show that $p^{k+1-n}j$ is in the image of the map $r\colon \pi_V SG\rightarrow \pi_{|V|-1}S_G$.

\begin{ex}\label{ex:c2root}
When $G = C_2$, one has $SC_2\simeq S^\sigma$, and thus we are asking for the restriction
\[
\res^{C_2}_e\colon \pi_{(8d-1)\sigma}S_{C_2}\rightarrow \pi_{8d-1}S
\]
to contain the top $\bbZ/(2)$ of the image of $J$ in this degree, say generated by $\alpha_d$. This is a consequence of the classical Mahowald invariant $\alpha_d \in M(2^{4d})$ \cite[Theorem 2.17]{mahowaldravenel1993root}. This connection between $C_2$-equivariant self maps and Mahowald invariants has been explored in detail by Behrens--Carlisle \cite[Section 8]{behrenscarlisle2024periodic}.
\tqed
\end{ex}

In general, consider the diagram
\begin{center}\begin{tikzcd}
\pi_V SG_p^\wedge\ar[r,"h_G"]\ar[d,"r"]&\pi_V(J_G \otimes SG)\ar[d,"r"]\ar[r,"\cong"]\ar[l,bend right,dashed]&\widetilde{J}_G^0(F(SG,S^V))\ar[d,"f"]\\
\pi_{|V|-1}S_p^\wedge\ar[r,"h"]&\pi_{|V|-1}J_p\ar[r,"\cong"]\ar[l,bend right,dashed]&\widetilde{J}_G^0(S^{|V|-1}\wedge G_+)
\end{tikzcd},\end{center}
where $J_G$ is the $G$-equivariant $J$-spectrum (see \cref{eq:jg}). By \cref{lem:suspension}, $F(SG,S^V)$ is equivalent to a $G$-equivariant suspension spectrum in such a way that $f$ is induced by a map of spaces. It follows from \cref{prop:adamsconj} that the horizontal maps are split surjections, yielding the dashed maps, and that these splittings are compatible with $r$. Moreover, $j \in \pi_{|V|-1}S \cong \pi_{|V|-1}S_p^\wedge$ is in the image of the splitting by construction.

We have therefore reduced verifying $p^{k+1-n}j\in \im(r)$ to verifying that this identity holds in $J_G$-theory. Running through the long exact sequence of \cref{eq:transferles} again, but in $J_G$-theory, this is equivalent to verifying that $p^{k+1-n}\tr_e^G(j) = 0$ in $\pi_{V-1} J_G$.

As $V$ is necessarily quaternionic of real dimension divisible by $8$, its complex Bott class lifts to $\beta_V \in \pi_V KO_G$. Write also $\beta_{|V|} = \res^G_e(\beta_V) \in \pi_{|V|}KO$. Consider the diagram
\begin{center}\begin{tikzcd}
\pi_V KO_G\ar[d,"\res^G_e",shift left=1mm]\ar[r,"\partial"]&\pi_{V-1}J_G\ar[d,"\res^G_e",shift left=1mm]\\
\pi_{|V|}KO\ar[r,"\partial"]\ar[u,"\tr_e^G",shift left=1mm]&\pi_{|V|-1}J_p\ar[u,"\tr_e^G",shift left=1mm]
\end{tikzcd},\end{center}
where $\partial$ is induced by the defining fiber sequences of $J_p$ and $J_G$. Up to a $p$-adic unit, which we can safely ignore, we can identify $j = \partial(\beta_{|V|})$. Frobenius reciprocity ensures that $\tr_e^G(\beta_{|V|}) = \tr_e^G(1)\beta_V$. Thus
\begin{align*}
p^{k+1-n}\tr_e^G(j) &= p^{k+1-n}\tr_e^G(\partial(\beta_{|V|})) \\
&= \partial(p^{k+1-n}\tr_e^G(\beta_{|V|})) \\
&= \partial(p^{k+1-n}\tr_e^G(1)\cdot \beta_V),
\end{align*}
and this vanishes in $\pi_{V-1}J_G \cong \coker\left(\psi^\ell-\psi^1\colon \pi_V KO_G\rightarrow \pi_V KO_G\right)$ by \cref{lem:adamsop}.

\subsection{Step 3}

We next verify that $0 \in \langle X,\alpha,X\rangle$. Indeed, combining \cref{lem:sq1} and \cref{lem:sq1int}, we can compute
\begin{align*}
\langle X,\alpha,X\rangle &\ni \Sq_1(X)\cdot \alpha \\
&= \Sq_1(X)\cdot p^{k+1-n-t}\tr_e^G(j) \\
&= p^{k+1-n-t}\tr_e^G(\res^G_e(\Sq_1(X))\cdot j)\\
&= p^{k+1-n-t}\tr_e^G(\Sq_1(p^{t}c)\cdot j) = \begin{cases} 2^{k-n}\tr_e^G(\eta\cdot j) & (p,t) = (2,1) \\
0 & \text{otherwise}. \end{cases}
\end{align*}
When $p>2$, we use the fact that we are working $p$-locally and so $\eta$ vanishes. When $(p,t) = (2,1)$, we have assumed that $k>n$, and thus $2^{k-n}\cdot\eta = 0$. In all cases we have $0 \in \langle X,\alpha,X\rangle$. This concludes the proof.

\section{Examples}\label{sec:examples}

We now enumerate examples of \cref{thm:mainthm2}, and give some further observations in the cyclic case. Everything in this section continues to be silently $p$-localized.

\subsection{Cyclic groups}

Consider first the cyclic $p$-group $C_{p^n}$. Recall from \cref{ssec:reptheory} that we have defined
\[
\rat_{p^n} = \Ind_{C_p}^{C_{p^n}}(\ol{\rho}_p^\bbC),
\]
where
\[
\ol{\rho}_p^\bbC = L+\cdots+L^{p-1}
\]
is the reduced complex regular representation of $C_p$, and that all fixed point free representations of $C_{p^n}$ with rational characters are a multiple of $\rat_{p^n}$. Hence \cref{thm:mainthm2} and \cref{cor:bottclass} combine to prove the following.

\begin{theorem}
There exists a self map
\[
\Sigma^{p^d\rat_{p^n}}C(p^s[C_{p^n}/C_{p^i}])\rightarrow C(p^s[C_{p^n}/C_{p^i}])
\]
inducing multiplication by $\beta_{\rat_{p^n}}^{p^d}$ in equivariant $K$-theory in the following cases:
\begin{enumerate}
\item $p>2$ and $d\geq s+n-i-1$;
\item $p=2$ and $d\geq \max(1,3-n,s+n-i-1)$.
\end{enumerate}
More generally, this holds with $p^s[C_{p^n}/C_{p^i}]$ replaced by any virtual $C_{p^n}$-set of cardinality $p^{s+n-i}c$ with $p\nmid c$.
\qed
\end{theorem}

\begin{ex}\label{ex:c22}
Set $\mathsf{h} = [C_2]$. Then there are $C_2$-equivariant $\beta_L = \beta_{\sigma_\bbC} = \beta_{\rat_2}$-self maps
\begin{align*}
\Sigma^{2^{\max(3,k)}\sigma}C(\mathsf{h}^k)&\rightarrow C(\mathsf{h}^k)\\
\Sigma^{2^{\max(3,k)}\sigma}C(2^{k})&\rightarrow C(2^{k})
\end{align*}
for $k\geq 1$, where $\sigma$ is the real sign representation of $C_2$.
\tqed
\end{ex}

\begin{ex}
There are $\beta_{L+L^3} = \beta_{\rat_4}$-self maps
\begin{align*}
\Sigma^{2(L+L^3)}C([C_4/C_2])&\rightarrow C([C_4/C_2]),\\
\Sigma^{2(L+L^3)}C(2[C_4/C_2])&\rightarrow C(2[C_4/C_2]),\\
\Sigma^{2(L+L^3)}C([C_4])&\rightarrow C([C_4]),
\end{align*}
and so on.
\tqed
\end{ex}

\begin{ex}
There are $\beta_{L+L^3+L^5+L^7} = \beta_{\rat_8}$-self maps
\begin{align*}
\Sigma^{2(L+L^3+L^5+L^7)}C([C_8/C_4])&\rightarrow C([C_8/C_4]),\\
\Sigma^{2(L+L^3+L^5+L^7)}C([C_8/C_2])&\rightarrow C([C_8/C_2]),\\
\Sigma^{4(L+L^3+L^5+L^7)}C([C_8])&\rightarrow C([C_8]),
\end{align*}
and so on.
\tqed
\end{ex}

\begin{ex}
There are $\beta_{\rat_{p^n}}$-self maps
\[
\Sigma^{p^n s\rat_{p^n}}C([C_{p^n}])\rightarrow C([C_{p^n}]),
\]
provided $s$ is even when $(p,n)= (2,1)$.
\tqed
\end{ex}

\begin{ex}
For $p$ odd, there are $\beta_{\ol{\rho}_p^\bbC} = \beta_{\rat_p}$-self maps
\[
\Sigma^{p^k(L+L^2+\cdots+L^{p-1})}C(p^k[C_p])\rightarrow C(p^k[C_p])
\]
for all $k \geq 0$. Here, one can also write $p^k[C_p] = [C_p]^{k+1}$.
\tqed
\end{ex}

By construction, a $\beta_{\rat_{p^n}}$-self map of order $p^d$ on $C(p^s[C_{p^n}/C_{p^i}])$ lifts the classical $v_1$-self map of order $p^{d+n-1}$ on $C(p^{s+n-i})$. We can also be reasonably explicit about their behavior on geometric fixed points. Recall that in general
\[
\Phi^{C_{m}}KU_{C_{m}} \cong KU[\tfrac{1}{m}](\zeta_{m}),
\]
where $\zeta_{m} = \Phi^{C_m}(L)$ is an $m$th primitive root of unity \cite[\S 7.7]{dieck1979transformation}.

\begin{lemma}[{\cite[Proposition 3.3.1]{balderramahouzhang2024cpn}}]\label{lem:geofixv1}
We have $\Phi^{C_{p^j}}(\beta_{\rat_{p^n}}) = p^{p^{n-j}}$.
\qed
\end{lemma}

\begin{theorem}\label{prop:typecofiber}
Let $v$ be a $\beta_{\rat_{p^n}}$-self map of order $p^d$ on $C(p^s[C_{p^n}/C_{p^i}])$ as above. Then
\[
\Phi^{C_{p^j}}(v) \equiv p^{p^{n-j+d}}\pmod{\text{nilpotents}}
\]
for $1\leq j \leq n$. In particular, the telescope of $v$ satisfies
\[
\Phi^{C_{p^j}}C(p^s[C_{p^n}/C_{p^i}])[v^{-1}]\simeq 
\begin{cases}
C(p^{s+n-i})[v_1^{-1}]&j=0,\\
0&1\leq j \leq i,\\
H\bbQ\oplus \Sigma H\bbQ & i < j \leq n,
\end{cases}
\]
and the geometric fixed point spectrum $\Phi^{C_{p^j}}C(v)$ of its cofiber has type exactly $1$ for $j\neq 0$.
\end{theorem}
\begin{proof}
Observe that
\[
\Phi^{C_{p^j}}C(p^s[C_{p^n}/C_{p^i}])\simeq
\begin{cases}
C(p^{s+n-i})&0\leq j \leq i,\\ 
S^0\oplus S^1&i<j\leq n\end{cases}.
\]
It therefore suffices to prove that $\Phi^{C_{p^j}}(v)$ is nilpotent for $1\leq j \leq i$ and equals $p^{p^{n-j+d}}$ after rationalization for $i < j \leq n$.

First consider the case $i < j \leq n$. As $v$ induces multiplication by $\beta_{\rat_{p^n}}^{p^d}$ in $KU_{C_{p^n}}$-theory and $\Phi^{C_{p^j}}KU_{C_{p^n}}\simeq KU[\tfrac{1}{p}](\zeta_{p^j})$, it suffices to prove that $\Phi^{C_{p^j}}(\beta_{\rat_{p^n}}^{p^d}) = p^{p^{n-j+d}}$, which is \cref{lem:geofixv1}.

Next consider the case $1\leq j \leq i$. To show that $\Phi^{C_{p^j}}(v)$ is nilpotent, we are free to replace $v$ with any power of $v$. Thus if $i=n$, then after possibly passing to a larger power of $v$ we can assume that $v$ is obtained by restriction from a $C_{p^{n+1}}$-equivariant $\beta_{\rat_{p^{n+1}}}$-self map on $C(p^{s-1}[C_{p^{n+1}}/C_{p^n}])$. In this way we may reduce to the case where $i\neq n$.

As $\Phi^{C_{p^j}}(v)$ is a self map $C(p^{s+n-i})\rightarrow C(p^{s+n-i})$, it is either nilpotent or an equivalence. If it were an equivalence, then $\Phi^{C_{p^j}}C(v)$, considered with its residual $C_{p^{n-j}}$-action, would be a compact $C_{p^{n-j}}$-spectrum whose underlying spectrum is contractible but whose geometric fixed points have nontrivial mod $p$ homology, by the case treated above. No such objects exist by classical Smith theory, and thus $\Phi^{C_{p^j}}(v)$ is nilpotent as claimed. 
\end{proof}

\begin{cor}
The telescope $C(p^s[C_{p^n}/C_{p^i}])[v^{-1}]$ is Bousfield equivalent to $KU_{C_{p^n}} \otimes C(p^s[C_{p^n}/C_{p^i}])$.
\end{cor}
\begin{proof}
Bousfield equivalences of $C_{p^n}$-spectra may be checked on geometric fixed points. As we are working $p$-locally we have $\Phi^{C_{p^j}}KU_{C_{p^n}} = KU_\bbQ(\zeta_{p^j})$ for $1\leq j \leq n$, and thus
\[
\Phi^{C_{p^j}}KU_{C_{p^n}} \otimes C(p^s[C_{p^n}/C_{p^i}]) \simeq \begin{cases} 
KU/(p^{s+n-i})&j=0,\\
0&1\leq j \leq i,\\
KU_\bbQ(\zeta_{p^j})\oplus \Sigma KU_\bbQ(\zeta_{p^j})&i<j\leq n.
\end{cases}
\]
Now the corollary follows from \cref{prop:typecofiber} and the nonequivariant height $1$ telescope conjecture, a theorem of Mahowald for $p=2$ \cite{mahowald1981bo} and Miller for $p>2$ \cite{miller1981relations}, which implies that $C(p^{s+n-i})[v_1^{-1}]$ is Bousfield equivalent to $KU/(p^{s+n-i})$.
\end{proof}

\subsection{Generalized quaternion groups} 

Let $Q_{2^n} = \Dic_{2^{n-2}}$ denote the generalized quaternion group of order $2^n$, defined for $n\geq 3$. Let $\bbH$ denote the irreducible symplectic representation of $Q_8$. Then all fixed point free complex $Q_{2^n}$-representations with rational characters are a multiple of the real $2^{n-2}$-dimensional representation $\Ind_{Q_8}^{Q_{2^n}}(\bbH)$. Specializing \cref{thm:mainthm2} to this case, we obtain the following. 

\begin{theorem}
Let $X$ be a virtual $Q_{2^n}$-set of cardinality $2^tc$ with $c$ odd. Then there exists a self map
\[
\Sigma^{2^{\max(2,t)}\Ind_{Q_8}^{Q_{2^n}}(\bbH)}C(X)\rightarrow C(X)
\]
inducing an equivalence in $Q_{2^n}$-equivariant $K$-theory.
\qed
\end{theorem}

\begin{ex}
There are $KU_{Q_8}$-equivalences
\begin{align*}
\Sigma^{2^{d+3}\bbH}C(2^d[Q_8]) &\rightarrow C(2^d[Q_8]),\\
\Sigma^{2^{d+2}\bbH}C(2^d[Q_8/C_2])&\rightarrow C(2^d[Q_8/C_2]),\\
\Sigma^{4\bbH}C([Q_8/\langle i \rangle])&\rightarrow C([Q_8/\langle i \rangle]),\\
\Sigma^{2^{d+2}\bbH}C(2^{d+1}[Q_8/\langle i \rangle])&\rightarrow C(2^{d+1}[Q_8/\langle i \rangle]),
\end{align*}
for $d\geq 0$.
\tqed
\end{ex}

\begingroup
\raggedright
\bibliography{refs}
\bibliographystyle{alpha}
\endgroup

\end{document}